\documentclass{amsart}
\usepackage{amsmath, amssymb, amsthm, amsfonts, amscd, hyperref}
\usepackage[all]{xy}

\newcommand{\QQ}{{\mathbb{Q}}}
\newcommand{\QQbar}{{\overline{\mathbb{Q}}}}
\newcommand{\CC}{{\mathbb{C}}}
\newcommand{\HH}{{\mathbb{H}}}

\newcommand{\cusps}{{\mathcal{C_N}}}

\newcommand{\RR}{{\mathbb{R}}}
\newcommand{\ZZ}{{\mathbb{Z}}}
\newcommand{\FF}{{\mathbb{F}}}
\newcommand{\TT}{{\mathbb{T}}}
\newcommand{\mm}{{\mathbf{m}}}

\newcommand{\dual}{\vee}
\newcommand{\goto}{\mapsto}

\newcommand{\isom}{\simeq}
\newcommand{\ndiv}{\nmid}
\newcommand{\tors}{\text{tors}}

\DeclareMathOperator{\num}{Num}

\DeclareMathOperator{\jac}{Jac}
\DeclareMathOperator{\End}{End}

\let\hom=\relax
\DeclareMathOperator{\hom}{Hom}
\DeclareMathOperator{\frob}{Frob}

\DeclareMathOperator{\id}{Id}

\DeclareMathOperator{\ord}{ord}

\DeclareMathOperator{\Otimes}{\bigotimes}

\newtheorem{thm}{Theorem}[section]
\newtheorem{lemma}[thm]{Lemma}
\newtheorem{prop}[thm]{Proposition}
\newtheorem{cor}[thm]{Corollary}

\theoremstyle{definition} 
\theoremstyle{remark} \newtheorem{remark}[thm]{Remark}
\theoremstyle{remark} 
\theoremstyle{remark}

\begin{document}
    \title{Modular Abelian Varieties of Odd Modular Degrees}
    \author{S. Yazdani}
    \address{McMaster University}
    \email{syazdani@math.mcmaster.ca}

    \thanks{This research was partially supported by NSERC.}
    \begin{abstract}
        In this paper, we will study modular Abelian varieties with 
        odd congruence
        numbers by examining the cuspidal subgroup of $J_0(N)$. We will
        show that the conductor of such Abelian varieties must be of
        a special type. For example, if $N$ is the conductor of an
        absolutely simple modular Abelian variety with an odd
        congruence number, then $N$ has at most
        two prime divisors, and if $N$ is odd, 
        then $N=p^\alpha$ or
        $N=pq$ for some prime $p$ and $q$. In the second half of this
        paper, we will focus on modular elliptic curves with odd modular
        degree.
        Our results, combined with
        the work of Agashe, Ribet, and Stein, finds necessary condition
        for elliptic curves to have odd modular degree. In the process
        we prove Watkins's conjecture for elliptic curves with odd
        modular degree and a nontrivial rational torsion point.
    \end{abstract}
    \maketitle

    Let $E/\QQ$ be an elliptic curve over the rational numbers. From the work of
Wiles, Taylor-Wiles, et al, we know that $E$ is modular (see \cite{BCDT}),
which implies that there is a surjective map $\pi:X_0(N) \rightarrow E$ 
defined over the rationals. As such, we have a new invariant attached to
the elliptic curve, namely the minimal degree of $\pi$, which we call the
{\em modular degree} of $E$. This invariant is related to many other invariants of an
the elliptic curve. For instance, this number is closely related to the congruences
between $E$ and other modular forms (see \ref{ssec:congnumber} and \cite{ARS}).
Also, we know that finding a good bound
on the degree of $\pi$ in terms of $N$ is equivalent to the $ABC$ conjecture
(see \cite{MM}, \cite{Frey}).

After calculating the modular degree of various elliptic curves, 
Watkins conjectured that
$2^r$ divides the modular degree of the elliptic curve $E$, 
where $r$ is the rank of $E(\QQ)$ 
(see \cite{Wat2000}).
In the particular case when the modular degree of $E$ is odd, Watkins's conjecture
implies $E(\QQ)$ is finite.
Searching through Cremona, Stein, and Watkins's database 
(\cite{SW} and \cite{CW}) for
elliptic curves of odd modular degree, Calegari and Emerton observed 
that all such elliptic curves have bad reduction at no more than two
primes. 
By studying the Atkin-Lehner involution on 
elliptic curves $E$ having odd modular degree,
they demonstrated that such curves have an even analytic rank and that
there are at most two odd primes dividing
their conductor (see \ref{sec22} and \cite{CE}). 
Dummigan has recently provided a heuristic explanation for
Watkins's conjecture.
His method uses the Selmer group of the symmetric
square of $E$ and its relationship to congruences between 
modular forms (see \cite{Dum}).

The goal of this paper is to generalize the results of Calegari and Emerton
to modular Abelian varieties having odd modular exponents and odd congruence
number (see \ref{ssec:congnumber} for definition).
We find necessary conditions for a modular Abelian variety 
to have an odd congruence number. Specifically in 
theorem \ref{modabclassification}
we show that if a modular Abelian variety with conductor $N$ has an odd
congruence number, then $N=2p, 4p^a, 8p^a,pq$ where $p$ and $q$
are odd primes, or $N$ is a power of a prime.
In section \ref{EllipticCurves} we study elliptic curves having odd
congruence numbers. 
Recall that the result of Agashe, Ribet, and
Stein, states that elliptic curves with semistable
reduction at $2$ have odd congruence number if and only if they have
odd modular degree (see theorem \ref{moddegcong}).
\footnote{In fact, by searching through Cremona table of elliptic curve, 
it seems that
an elliptic curve has an odd congruence number if and only if it has an
odd modular degree.}
We find more stringent conditions that elliptic
curves with an odd congruence number need to satisfy. Specifically if an 
elliptic curve $E$ with conductor $N$ has an odd congruence number, then 
if it has a 
trivial torsion structure then $N$ is prime and $E$ has an even analytic rank, 
otherwise $N$ has at most two prime divisors and has rank $0$.
Furthermore, we find families of elliptic curves that any elliptic curve
with odd congruence number and a non-trivial torsion point must belong to
one of these families (see theorem \ref{oddellipticcurves}). 
We expect that the
elliptic curves in these families have odd modular degrees, although
to prove this we need a better understanding of the rational torsion points of
$J_0(N)$.

We now give a quick overview of this article.
In section \ref{Prelims}, we review some of the definitions used in this paper,
along with some results that come in handy in the rest of the paper.
Specifically, in section \ref{Cuspidal} we recall how to calculate the
rational cuspidal subgroup of $J_0(N)$, and in section \ref{sec:HeckeAction}
we study the action of the Hecke algebra and Atkin-Lehner involutions
on this subgroup.
In section \ref{ModularAV}, we study modular
Abelian varieties with odd congruence numbers, and show that all such
Abelian varieties have at most two primes of bad reduction. A key component
of this argument is that if $A$ is a modular Abelian variety having 
non prime-power conductor 
and if $A$ has an odd congruence number, then it must have a 
rational $2$-torsion point (theorem \ref{thmCE1}). 
We also show that if $A$ has an odd congruence number and a rational 
$2$-torsion
point, then all of the new rational $2$-torsion points of $J_0(N)$
map injectively to $A$ (see section \ref{sec23}). We use this fact and our 
analysis of cuspidal subgroup to show that if $A$ has an odd congruence number
and is semistable away from $2$, then it has at most two primes of bad reduction
(theorem \ref{CongNumberPQ}) and the primes dividing the conductor must
satisfy certain congruences (theorem \ref{modabclassification}).
The other useful result is that if $p^2|N$
for some odd prime $N$, then $A$ must have a complex multiplication or an
inner twist (section \ref{CMCase}).
In section
\ref{EllipticCurves} we apply our results 
to elliptic curves. Theorem \ref{modabclassification} gives us different
type of conductors that elliptic curves with odd congruence number must 
satisfy. In each subsection of section \ref{EllipticCurves} we study
one of these cases, and get more stringent conditions on the conductor,
and show that in almost all cases the rank of such elliptic curves is zero
(theorem \ref{oddellipticcurves}).

{\bf Acknowledgements:}
This paper would not have been possible without the help of my advisor,
Ken Ribet. Specifically, many of the results in section 2.4 were suggested
to me by him. I would also like to thank Frank Calegari, Matt Emerton, William
Stein, and Jared Weienstein, with whom I have had many discussions. 
Manfred Kolster and Romyar Sharifi gave me very useful feedback 
on the first draft of this article. Finally, I would like to thank Jovanca
Buac for her careful reading of this paper and all of her suggestions.

\section{Preliminaries}
\label{Prelims}
Let $N$ be a positive integer and $X_0(N)$ be the moduli space of
elliptic curves with a cyclic subgroup of order $N$.
Let $\cusps \subset X_0(N)$ be the
set of cusps of $X_0(N)$, that is $\cusps= \pi^{-1}(\infty)$,
where $\pi:X_0(N) \rightarrow X_0(1)$ is the natural degeneracy map,
and $\infty$ is the unique cusp on $X_0(1)$. 
All such cusps can be represented as rational numbers
${a \over b} \in \HH$,
with $a$ and $b$ positive coprime integers and $b | N$. 
Furthermore, there is a unique representative for any cusp with
$a \leq (b,N/b)$.
Under this representation, $\infty = {1\over N}$.
For any integer $r | N$ such that $\gcd(r,N/r)=1$,
we can define the Atkin-Lehner involution $w_r:X_0(N) \rightarrow X_0(N)$,
by sending $(E,D) \in X_0(N)$ to $(E/D[r],(E[r]+D)/D[r])$.\footnote{As usual, $G[r]$ is the set of $r$-torsion points of the group $G$.}
We usually abuse notation by letting $w_{\overline{r}}=w_r$ whenever
$\overline{r}=\prod_{l | r} l$ (for example $w_4=w_2$ on $X_0(4N)$).

Let $S(N)$ be the space of weight two cuspforms on
$\Gamma_0(N)$. Let $\TT$ denote the $\ZZ$-algebra of the Hecke operators
acting on $S(N)$.
As usual, we denote $J_0(N)=\jac(X_0(N))$. Then, $\TT$ acts faithfully on 
$J_0(N)$ by Picard functoriality.
We also have the standard
Albanese embedding $i:X_0(N) \rightarrow J_0(N)$ via $i(z)=(z)-(\infty)$.
Note that for any map $w:X_0(N) \rightarrow X_0(N)$ we have
the induced map
\begin{eqnarray*}
    w_* : & J_0(N)  \rightarrow & J_0(N) \\
    & \sum (z)  \goto & \sum( w(z)).
\end{eqnarray*}

\subsection{Congruence Numbers}
    \label{ssec:congnumber}
    Recall that attached to any newform $f \in S(N)$ we have a
    modular Abelian variety $A_f$. Specifically, let $I_f$ be the
    kernel of $\TT \rightarrow \CC$ induced by $f$. Then we have
    $A_f=J_0(N)/I_f$, which we refer to as the {\em optimal quotient}
    attached to $f$. Conversely, if $A$ is a simple
    quotient of $J_0(N)$ that is stable under the action of $\TT$
    and the Atkin-Lehner involutions, then we can find a modular eigenform
    $f \in S(N)$ such that $A$ is isogenous to $A_f$. In this case,
    we say that $f$ is attached to $A$. Furthermore
    all modular forms attached to $A$ are Galois conjugate to $f$.
    Let $\phi:J_0(N) \rightarrow A$ be a surjective morphism. 
    Then the dual morphism is
    $\phi^\dual:A^\dual \rightarrow J_0(N)^\dual.$
    Since $J_0(N)$ is self dual, we can compose these two morphisms to get
    \[ \psi : A^\dual \rightarrow A.\]
    Following \cite{ARS}, define {\em modular number} to be the order
    of $\ker(\psi)$, and {\em modular exponent} to be its exponent, denoted
    by $\widetilde{n_A}$.
    If $A$ is an elliptic curve, then $\widetilde{n_A}$ equals to
    the modular degree of $A$. In fact, in the case of elliptic curves 
    we get that $\ker(\psi)=A[\deg(\pi)]$ where $\pi:X_0(N) \rightarrow A$
    (see lemma \ref{mult}).

    Now let $\phi:J_0(N)\rightarrow A$ be any optimal modular Abelian quotient.
    Let $B=\ker(\phi)$, which is an Abelian variety since $A$ is 
    an optimal quotient. Let $\TT_A$ be the $\ZZ$-algebra of the Hecke operators
    acting on $A$. Similarly, let $\TT_B$ be the $\ZZ$-algebra of the Hecke 
    operators acting on $B$. 
    There is an injective map $\TT \rightarrow \TT_A \oplus \TT_B$ with a
    finite index, given by the restriction map.
    The order of the cokernel of $\TT \rightarrow \TT_A \oplus \TT_B$ is 
    the {\em congruence number} of $A$. The exponent of this cokernel
    is the {\em congruence exponent} of $A$, which is 
    denoted by $\widetilde{r_A}$ (see lemma 4.3 of \cite{ARS}).
    Let $\mm \subset \TT$ be a maximal ideal of $\TT$. Then $A[\mm] \neq 0$
    (resp. $B[\mm] \neq 0$)
    if and only if image of $\mm$ in $\TT_A$ (resp. $\TT_B$) is a proper
    maximal ideal. If $A[\mm]$ and $B[\mm]$ are both nontrivial, then
    by tensoring $\TT \rightarrow \TT_A \oplus \TT_B$ by $\TT/\mm$, we 
    see that the cokernel
    is a nontrivial vector space over $\TT/\mm$, which means that
    the characteristic of $\TT/\mm$ divides the congruence exponent of $A$.
    On the other hand, if $A[\mm] \neq 0$, then $A^\vee[\mm] \neq 0$,
    and if $A[\mm] \cap B[\mm] \neq 0$, then the characteristic of 
    $\TT/\mm$ divides the modular exponent.
    

    In \cite{ARS}, the relationship between the modular exponent
    and the congruence exponent was studied, and the following was proved.
    \begin{thm}
        \label{moddegcong}
        If $f \in S(N)$ is a newform, then
        \begin{enumerate}
            \item $\widetilde{n_{A_f}} | \widetilde{r_{A_f}}$, and
            \item if $p^2 \ndiv N$, then $\ord_p(\widetilde{n_{A_f}})=\ord_p(\widetilde{r_{A_f}}).$
        \end{enumerate}
    \end{thm}
    In particular, if $f$ is a newform of level $N$ and $4 \ndiv N$, then
    the modular exponent of $A_f$
    is odd if and only if its congruence exponent is odd.

\subsection{Cuspidal Subgroup}
\label{Cuspidal}
The cuspidal subgroup of $J_0(N)$ is the subgroup generated by the 
cusps of $X_0(N)$. 
The goal of this section is to understand the rational points of
the cuspidal subgroup of $J_0(N)$, denoted by $C_N$.
This problem is studied for
$N$ a power of a prime by San Ling \cite{SanLing1} 
and for $N$ the product of the two primes by
Seng-Kiat Chua and San Ling \cite{SanLing2}.
Following \cite{SanLing1}, let 
\[P_d={1\over \gcd(d,N/d)}\sum_{i=1}^{\gcd(d,N/d)} (id/N). 
\footnote{Our notation is slightly different from San Ling's papers. 
Specifically San Ling's $P_d$ is $\gcd(d,N/d)$ times our $P_d$.}
\]
With this notation we get
\begin{prop}
    The rational {\em cuspidal subgroup} $C_N \subset J_0(N)$ is 
    generated by the elements $\phi(\gcd(d,N/d))(P_d-P_1)$,
    where $\phi(k)$ is the Euler $\phi$-function.\footnote{This is the
    only section were $\phi$ is the Euler function. For the rest of the
    paper, whenever results of this section are used, $\phi(\gcd(d,N/d))=1$.
    Also outside of this section, $\phi$ is reserved for the map 
    $\phi:J_0(N) \rightarrow A$.}
    \label{CuspidalDefn}
\end{prop}
\begin{proof}
    See \cite{SanLing1}.
\end{proof}
In this subsection, we calculate the order of certain elements in this
group.
Recall that Dedekind's eta function is defined as
$$\eta(\tau)=q^{1/24}\prod_{n=1}^\infty (1-q^n),$$ where $q=e^{2\pi i \tau}.$
Let $\eta(M\tau)=\eta_M(\tau)$.
We use $\eta_M$ to construct 
functions with divisors supported on the cusps. 
In particular, for $M|N$, $\eta_M$ has a zero of order 
\begin{equation}
    {1\over 24} {N d'^2 \over dt M},
    \label{CuspIsom1}
\end{equation}
at the cusp of $X_0(N)$ corresponding to $x/d \in \HH$,
where $d'=\gcd(d,M)$ and $t=\gcd(d,N/d)$ (see, for example, \cite{Ogg74}). 
The following result of Ligozat can be used to calculate the order
of specific cusps.
\begin{prop}
    Let ${\bf r}=(r_\delta)$ be a family of rational numbers $r_\delta \in \QQ$
    indexed by all of the positive divisors of $\delta | N$. Then the function
    $g_{\bf r}=\prod_{\delta | N} \eta_\delta^{r_\delta}$ is a modular function
    on $X_0(N)$ if and only if the following conditions are satisfied:
    \begin{enumerate}
        \item All of the rational numbers $r_\delta$, are rational integers;
        \item $\sum_{\delta | N} r_\delta \delta \equiv 0 \pmod {24}$;
        \item $\sum_{\delta | N} r_\delta{N \over \delta} \equiv 0 \pmod {24}$;
        \item $\sum_{\delta | N} r_\delta=0$; 
        \item $\prod_{\delta | N} \delta^{r_\delta}$ is a square of a rational number.
    \end{enumerate}
    \label{ligozatprop}
\end{prop}
\begin{proof}
    See \cite{Ligozat75}.
\end{proof}
We also know that the lattice of divisors linearly equivalent to zero supported
on the cusps is generated by the divisors of $g_{\mathbf r}$ that are 
modular functions. 
Let $N=\prod_{i=1}^k p_i^{s_i}$ be the prime factorization of $N$, 
and let $V$ be the rational vector space
spanned by $P_d$ for $d|N.$ We can represent this vector space as the tensor
product of the vector spaces $V_{p_i}$
where $V_{p_i}$ is the $(s_i+1)$-dimensional space generated by 
$P_{1},P_{p_i},\ldots,P_{p_i^{s_i}}$.
(The isomorphism between $V$ and the tensor product $\Otimes_{i} V_{p_i}$
is the natural one sending $P_{\prod p_i^{\alpha_i}}$ to 
$\Otimes_i P_{p_i^{\alpha_i}}$.)
Similarly, let $W$ be rational vector space of functions $g_{\bf r}$
(as defined in proposition \ref{ligozatprop}) under multiplication.
Then we have $W \isom \Otimes W_{p_i}$ where $W_{p_i}$ is the 
$(s_i+1)$-dimensional vector space generated by 
$\eta_1,\eta_{p_i},\ldots,\eta_{p_i^{s_i}}.$
We have an isomorphism $\Lambda : V \rightarrow W$ 
where $\Lambda^{-1}(g)$ is the divisor
attached to $g$.
We can verify that this isomorphism can be written very
explicitly as
\begin{equation*}
    24 \Otimes_{p_i} \Lambda_{p_i}.
\end{equation*}
where $\Lambda_{p_i} : V_{p_i} \rightarrow W_{p_i}$
and $\Lambda_{p_i}$ is the tridiagonal matrix (under the above basis)
\begin{equation*}
    \Lambda_{p_i}=
    {1 \over (p_i^2-1)\phi(p_i^{s_i})}
    \begin{pmatrix}
        p_i(p_i-1)   & -p_i         &    &   & \\
        -(p_i-1)   & p_i^2+1 & -p_i &   & \\
       & -p_i & p_i^2+1 & -p_i &  &  \\
       & & \ddots & \ddots & \ddots &  \\
       & & & -p_i & p_i^2+1 & -(p_i-1) \\
       & & & & -p_i & p_i(p_i-1)
    \end{pmatrix}.
\end{equation*}
Note that when $f \in W$ is a modular function, $\Lambda^{-1}(f)$ is 
linearly equivalent to zero.
Therefore, by combining proposition \ref{ligozatprop} and the above isomorphism we
get
\begin{prop}
    \label{ligozat2}
    An element $v \in \Otimes V_{p_i} = V$ is linearly equivalent to zero if
    the following conditions are satisfied:
    \begin{enumerate}
        \item All of the coefficients in $\Lambda v$ are integral;
        \item $v$ has degree $0$;
        \item $v$ is integral and the coefficient of $P_d$ divides
            $\phi(d,N/d)$;
        \item Let $e_i=(1,1,1,\ldots,1) \in W_{p_i}^\dual$ and 
            $f_i=(0,1,0,1,\ldots) \in W_{p_i}^\dual.$ Then 
            for each $i$,
            \[ (e_1\otimes \cdots \otimes f_i \otimes \cdots \otimes e_k)\Lambda v \]
            is an even number.
    \end{enumerate}
\end{prop}
\begin{proof}
    This is a straightforward rewording of proposition \ref{ligozatprop}.
\end{proof}
We use proposition \ref{ligozat2} to calculate the order of the elements in 
$C_N$.
Specifically, for an integral element $v \in V$ of degree zero, the order of
$v$ in $C_N$ is the smallest positive integer $n$ such that $nv$ satisfies all
the conditions in proposition \ref{ligozat2}.
Notice that if $N=2^{s_2}M$
where $M$ is square free odd integer and $s_2<4$ 
(the case we come across in this paper), 
then condition three is reduced to
the coefficients of $v$ being integral.
Therefore, the denominator of $\Lambda(v)$ gives the order
of $v$ or half of the order of $v$. 

We use the above proposition to calculate the order of various cusps:

\begin{tabular}{|c|c|c|l|}
    \hline
    $N$ & Cusp & Order & Conditions \\
    \hline \hline
    $p$ & $P_1-P_p$ & $\num\left({p-1 \over 12}\right)$ & \mbox{ }\\ \hline
    $\prod_{i=1}^t p_i$ & $\Otimes_{i} (P_1+b_iP_{p_i})$ & 
       $\num\left( \prod_i(p_i+b_i) \over 24 \right)$  &
           $t>1$, \\
           & & & $b_i = \pm 1$ for $i=1,2,\ldots, t$, \\
           & & & $b_j=-1$ for at least one of the $j$'s. \\ \hline
    $4p$ & $P_2 - P_{2p}$ & ${p-1 \over 2}$ & $p$ is odd. \\ \hline
    $4 \prod_{i=1}^t p_i$ & $P_2 \otimes \Otimes_i (P_1 + b_iP_{p_i})$ & 
        $\left( \prod_i (p_i+b_i) \over 4 \right)$ &
            $t>1$, \\
            & & & $p_i$'s are all odd, \\
            & & & $b_i = \pm 1$ for $i=1,2,$\ldots,$t$, \\
            & & & $b_j=-1$ for at least one of the $j$'s. \\  \hline
    $8 \prod_{i=1}^t p_i$ & $(P_1-P_8) \otimes \Otimes_i (P_1+b_iP_{p_i})$ & 
        ${\prod_i p_i+b_i \over 2}$ & $p_i$'s are odd. \\ \hline
\end{tabular}

As an example of the details of calculating the order, consider the 
element $z = \Otimes_i (P_1+b_i P_{p_i}) \in J_0(N)$ with $N$ square free
and not a prime. This is a generalization of
the work of Ogg \cite{Ogg74} in the case where $N=pq$.
Note that 
\begin{eqnarray*}
    \Lambda z &=&{24 \over \prod_i (p_i^2-1)(p_i-1)} \left(\Otimes_i
            \begin{pmatrix} p_i(p_i-1) & -(p_i-1) \\ -(p_i-1) & p_i(p_i-1) \end{pmatrix} 
            \begin{pmatrix} 1 \\ b_i \end{pmatrix}\right) \\
        &=& {24 \over \prod_i (p_i^2-1)} \left( \Otimes_i \begin{pmatrix} p_i-b_i \\ (p_i-b_i)b_i \end{pmatrix} \right)\\
        &=& {24 \over \prod_i (p_i+b_i)} \left( \Otimes_i \begin{pmatrix} 1 \\ b_i \end{pmatrix} \right).
\end{eqnarray*}
Considering the coefficient of the first coordinate, the order is at least
$n=\num\left( \prod_i (p_i+b_i)/24 \right)$. On the other hand, 
$n\Lambda z=\Otimes_i \begin{pmatrix} 1 \\ b_i \end{pmatrix}$. Therefore
\[ (e_1\otimes \cdots \otimes f_i \otimes \cdots \otimes e_t) (n \Lambda z),\]
is even, which implies that $n \Lambda z$ is trivial. Therefore the order of
$z$ is
\[ \num\left({\prod (p_i+b_i) \over 24}\right).\]

\subsection{Hecke Action}
\label{sec:HeckeAction}
In this section we recall the explicit action of the Hecke operators
$T_l$ on the rational cuspidal divisors of $X_0(N)$. This is fairly standard,
although the representation of these actions 
as the tensor product of matrices is not that common. 
The following is the main result of this section.
\begin{prop}
    \label{prop:HeckeAction}
    \begin{enumerate}
        \item Let $p \ndiv N$. Then $T_p:V \rightarrow V$ acts as
            multiplication by $p+1$.
        \item Let $p | N$ and $V=\Otimes V_{p_i}$. Then $T_p$ acts
            trivially on $V_{p_i}$ for $p_i \neq p$, and as
            \[ 
                \begin{pmatrix}
                    1 & 0 & \cdots & 0 & 0 \\
                    p-1 & 0 & \cdots & 0 & 0\\
                    0 & p & \cdots & 0 & 0 \\
                    \vdots & \vdots & \ddots & \vdots \\
                    0 & 0 & \cdots & 0 & 0 \\
                    0 & 0& \cdots & p & p
                \end{pmatrix}
            \]
            on $V_p$ with the standard basis,
            where the diagonal elements are all $0$ except for the first
            and last one,
            while the sub-diagonal elements are all $p$, except for the
            first one.
        \item For $p | N$ we have $w_p$ acting trivially on $V_{p_i}$ for
            $p_i \neq p$, and as
            \[ 
                \begin{pmatrix}
                    0 & 0 & \cdots & 1 \\
                    \vdots & \vdots & \ddots & \vdots \\
                    0 & 1& \cdots & 0 \\
                    1 & 0& \cdots & 0
                \end{pmatrix} : V_p \rightarrow V_p.
            \]
    \end{enumerate}
\end{prop}
We will omit the proof of this proposition.

\begin{remark}
    Applying $w_2$ to $P_2$ when $N=4M$ with $M$ odd,
    we see that $w_2$ has a fixed point on $X_0(4M)$.
\end{remark}
We can use this explicit formula to calculate the action of $T_p$ for
various elements in the cuspidal subgroup.
\begin{prop}
    \label{newness}
    Let $M=\prod p_i$ be an odd square free integer and
    $N=2^aM$ for some $a<4$.
    Let $v=\Otimes v_{l}$
    be an element in the cuspidal subgroup. 
    Then
    \begin{enumerate}
        \item If $p || N$ and $v_{p}=P_1-P_{p}$ then $T_p v = v$.
        \item If $p || N$ and $v_{p}=P_1+P_p$ then $T_p v = v+2u$
            where $u=\Otimes u_{l}$ with $u_{l}=v_{l}$ for all 
            $l \neq p$ and $u_p=(p-1)P_p$.
        \item If $N=4M$ and $v_2= P_2$ then $T_2 v = u$ with
            $u= \Otimes u_{l}$ with $u_2=2P_4$ and $u_{l}=v_{l}$
            for all odd $l$.
        \item If $N=8M$ and $v_2=P_1-P_8$ then $T_2 v = u$ where
            $u= \Otimes u_{l}$ with $u_2=P_1+P_2-2P_4$ and
            $u_l=v_l$ for all odd $l$.
    \end{enumerate}
    Specifically, in all of the cases above, if $\lambda v$ is of order
    $2$ for some integer $\lambda$, then 
    $T_{p} (\lambda v)=\lambda v$ for all odd $p|M$ and
    $T_2(\lambda v)=\lambda v$ (resp. $T_2(\lambda v)=0$) when
    $N=2M$ (resp. $N=4M$ or $N=8M$).
\end{prop}
\begin{proof}
    Calculating the action of various Hecke operators on the above
    elements is a straight forward matrix multiplication.
    As for proving $T_{p} (\lambda v)=\lambda v$ when $N$ is square free,
    case one follows by definition.
    In second case (when $v_p=P_1+P_p$), we can verify that 
    $u$
    has the same order as $v$, hence $2\lambda u=0$.
    As for the cases $N=4M$ or $N=8M$, we can check that order of
    $T_2 v$ is half of the order $v$, hence $T_2(\lambda v)=0$.
\end{proof}
Recall that if $A$ is a simple new modular form, then for $p || N$,
$T_p|_A$ is acting as either $1$ or $-1$, and when 
$p^2 | N$ then $T_p|_A=0$. Hence, the above proposition is finding
explicit $2$-torsion points of $C_N$ that are new.
This will be used to create congruences between modular forms in 
later sections.

    \section{Modular Abelian Varieties with Odd Congruence Number}
\label{ModularAV}
In this section we will study simple modular Abelian varieties with odd congruence
numbers. By examining the twists of modular Abelian varieties, the action of the 
Atkin-Lehner involutions, and the order of the cuspidal subgroup, we 
demonstrate that if 
we have 
an absolutely simple modular Abelian variety with an odd congruence number, then
its conductor $N$ has at most two prime divisors. We also show that the
odd part of $N$ is either square free or a power of a prime, and if $16 | N$,
then $N$ is a power of $2$. Furthermore, we find some congruences that prime divisors
of $N$ must satisfy.

Throughout this section we let $A$ be an optimal modular Abelian variety with conductor
$N$ and we fix a surjective map $\phi:J_0(N) \rightarrow A$ defined over
$\ZZ[1/N]$. Furthermore, let $\pi: X_0(N) \rightarrow A$
be the composition of the Albanese embedding and $\phi$.
As usual, let $\TT$ be the Hecke algebra acting on $J_0(N)$ and $S(N)$.

\subsection{Atkin-Lehner Involution}
    \label{sec22}
    The goal of this section is to prove the following
    \begin{thm}
        Let $A$ be a new simple modular Abelian variety with an odd 
        modular exponent.
        Then if $A(\QQ)$ has no $2$-torsion points, then the conductor
        of $A$ is a power of a prime. Furthermore if $A$ has good reduction
        at $2$ and $A(\FF_2)$ has no $2$-torsion points, then the conductor
        of $A$ is a power of a prime.
        \label{thmCE1}
    \end{thm}
    This theorem was proved by Calegari and Emerton in the case where $A$ is
    an elliptic curve (theorem 2.1 of \cite{CE}). 
    Here, we apply their techniques to higher dimensional 
    modular Abelian varieties. We must prove a few lemmata first.

    \begin{lemma}
        Let $k$ be a field and 
        $f:X/k \rightarrow Y/k$ be a degree $m$ map between curves.
        Then the composition
        \[\xymatrix@1{
            \jac(Y) \isom \jac(Y)^{\vee}  \ar[r]^-{f^*} & \jac(X)^\vee \isom 
            \jac(X) \ar[r]^-{f_*} &  \jac(Y) 
        }\]
        is multiplication by $m$.
        \label{mult}
    \end{lemma}
    \begin{proof}
        It suffices to verify the above lemma for the points
        $(z_1) - (z_2) \in \jac(Y),$
        since these points generate $\jac(Y)$. Unraveling the definitions
        we get
        \begin{eqnarray*}
            f_*(f^*( (z_1)-(z_2))) &=&  f_*\left( \sum_{f(y_1)=z_1} (y_1)- \sum_{f(y_2)=z_2} (y_2) \right) \\
            &=& \left( \sum_{y_1\in f^{-1}(z_1)} (z_1) -\sum_{y_2=f^{-1}(z_2)} (z_2) \right) \\
            &=& m( (z_1)-(z_2) )
        \end{eqnarray*}
        where the summations are understood to account for multiplicities.
    \end{proof}

    \begin{lemma}
        Let $w$ be an involution on $X_0(N)$.
        Assume that
        \[ \xymatrix{
            X_0(N) \ar[dd]_w \ar[rd]^\pi & \\ & A \\
            X_0(N) \ar[ru]^\pi & 
        } \]
        commutes.
        Then the modular exponent of $A$ is even.
        \label{lemCE}
    \end{lemma}
    \begin{proof}
        The above assumptions imply that $\pi$ factors through
        \[ \xymatrix{ X_0(N) \ar[r] & X_0(N)/w \ar[r] & A }. \]
        Therefore $\phi$ factors through
        \[ \xymatrix{ \jac(X_0(N)) \ar[r] & \jac(X_0(N)/w) \ar[r] & A }. \]
        Dualizing the above diagram and using the autoduality of $J_0(N)$, we get
        \[ \xymatrix{
        A^\vee \ar[r] \ar@{.>}[d]_\delta & \jac(X_0(N)/w)^\vee  \ar[r] \ar@{.>}[d] & J_0(N)^\vee \ar[d] \\ 
        A & \jac(X_0(N)/w) \ar[l] &  J_0(N) \ar[l] .
        } \]
        By lemma \ref{mult}, the middle arrow
        is multiplication by $2$, since the degree of $X_0(N) \rightarrow X_0(N)/w$
        is $2$. Using the commutativity of the above diagram, we can see that 
        $A^\vee[2] \subset \ker(\delta)$.
        Recalling that the modular exponent is the exponent of the 
        kernel of $\delta$, we conclude that the modular exponent
        of $A$ is even.
    \end{proof}
    Recall that for an involution map $w:X_0(N) \rightarrow X_0(N)$, we get
    the induced map $w_* : J_0(N) \rightarrow J_0(N)$.
    Let $A$ be an optimal modular Abelian variety, and
    $\phi:J_0(N) \rightarrow A$ the associated surjective map.
    Then if $w_*$ keeps $\ker(\phi)$ invariant, then $w_*$ acts
    on $A$ as well (this happens when, for example, $w$ is an
    Atkin-Lehner involution and $A$ is new). 
    The following lemma deals with the case when $w_*$ is trivial on $A$.
    \begin{lemma}
    \label{twotorsionlemma}
        Let $k$ be either $\QQ$ or $\FF_p$ with $p \ndiv N$.
        Let $A$ be an optimal modular Abelian variety with an 
        odd modular exponent. 
        As before let $\pi:X_0(N) \rightarrow A$ be the composition
        of Albanese embedding $X_0(N) \rightarrow J_0(N)$ and $\phi$.
        Assume that for some involution $w$,
        $w_*:J_0(N)\rightarrow J_0(N)$ descends down to a trivial action
        on $A$.
        Then $\pi(w(z))-\pi(z)$ is a nontrivial $k$-rational
        $2$-torsion point for all $z\in X_0(N)(\overline{k}).$ 
    \end{lemma}
    \begin{proof}
        Recall that $P_1$ is the cusp at infinity and $\pi(z)=\phi(z-P_1)$.
        Then we get
        \begin{eqnarray*}
            \pi(w(z))-\pi(z)&=&\phi(w(z)-P_1)-\phi(z-P_1) \\
            &=& \phi(w(z)-w(P_1))-\phi(z-P_1)+\phi(w(P_1)-P_1) \\
            &=& w_*(\phi(z-P_1))-\phi(z-P_1)+\phi(w(P_1)-P_1) \\
            &=& \pi(w(P_1)).
        \end{eqnarray*}
        Therefore $\pi(w(z))=\pi(z)+\pi(w(P_1))$ for all $z \in X_0(N)$.
        Applying this equation to $w(z)$ we get
        $\pi(w(w(z)))=\pi(w(z))+\pi(w(P_1))=\pi(z)+2\pi(w(P_1)).$
        Therefore, $2\pi(w(P_1))=0$. 
        By lemma \ref{lemCE}, 
        if $A$ has an odd modular exponent, then
        $\pi(w(z))-\pi(z)$ is nontrivial. 
        Thus, $\pi(w(P_1))$ is a nontrivial $2$-torsion point of $A$.
        It is $k$ rational because $w(P_1)$ is also $k$ rational.
    \end{proof}

    Given the above lemma, we can now prove theorem \ref{thmCE1}.
    \begin{proof}
        Let $W$ be the group of Atkin-Lehner involutions on $X_0(N)$,
        and let $k=\QQ$ or $\FF_2$ when $N$ is odd.
        Since we are assuming that $A$ is new and simple, for any Atkin-Lehner
        involution $w \in W$, we have $w_*(z)=\pm z$ for all 
        $z \in A(\overline{k})$. 
        This gives us a map $W \rightarrow \{ \pm 1 \}.$ Let $W_0$ be the
        kernel of this map.
        Note that $W_0$ has index at most $2$ in $W$.
        Assume that $N$ is not a power of a prime, hence $W$ will have more
        than $2$ elements.
        Therefore, we can find a 
        non-trivial element $w \in W_0$, that is $w_*(z)=z$ for all 
        $z \in A(\overline{k})$. Applying lemma
        \ref{twotorsionlemma}, we find that 
        $0 \neq \pi(w(P_1)) \in A[2](k).$
        Therefore, if $A[2](k)=0$ then $N$ must be a power of a prime.
    \end{proof}
    Lemma \ref{twotorsionlemma} can also be used to find the signs of the
    Atkin-Lehner involutions on $A$ in certain cases.
    \begin{lemma}
        Let $A$ be a new simple modular simple Abelian variety with \
        conductor $N$ and an odd modular exponent. If
        the Atkin-Lehner involution $w_r:X_0(N) \rightarrow X_0(N)$ has a fixed
        point then $(w_r)_*$ acts as $-1$ on $A$. 
        Specifically, $(w_N)_*$ acts as $-1$ on $A$.
        When $N=2M$ (resp. $N=4M$), $(w_2)_*$ acts as $1$ 
        (resp. $(w_2)_*$ acts as $-1$) on $A$.
        \label{AtkinLehnerSign}
    \end{lemma}
    \begin{proof}
        Let $P \in X_0(N)(\overline{\QQ})$ be the fixed point of $w_r$. 
        Then $\pi(P)=\pi(w_r(P))$,
        which implies that $\pi(w_r(P))-\pi(P)=0$. However, we know that if
        $(w_r)_*=1$ then $\pi(w_r(z))-\pi(z)=\pi(w_r(P_1))$ for any 
        $z \in X_0(N)(\overline{\QQ})$.
        Specifically, we get $\pi(w_r(z))=\pi(z)$, which by lemma \ref{lemCE} implies
        that $A$ has an even congruence number. Therefore $(w_r)_*=-1$ when $w_r$ has a
        fixed point in $X_0(N)$.
        
        Finally, the point $\sqrt{-N}$ is fixed by $w_N$.
        When $N=2M$, we can check that ${1\over M-i\sqrt{M}}$ is fixed under
        $(w_M)_*$. Similarly, when $N=4M$, $P_2$ is fixed under $(w_2)_*$.
        Therefore, we have the desired result.
    \end{proof}
    Since $(w_N)_*$ is the sign of the functional equation, we get the following
    \begin{cor}
        \label{evenrank}
        If $A$ is a simple modular Abelian variety with an odd congruence number, then
        the analytic rank of $A$ is even.
    \end{cor}
    \begin{remark}
        Calegari and Emerton used theorem \ref{thmCE1} for modular elliptic curves
        $E$ with odd modular degree and conductor $N$ to show that $N$ has at 
        most two odd prime divisors.
        Specifically, since
        $E[2](\QQ)$ has at most $4$ elements, an immediate corollary of theorem \ref{thmCE1}
        is that if $N$ has more than $3$ prime divisors, then $E$ has even modular
        degree. Similarly, if $E$ has good reduction at $2$, then since 
        $E[2](\FF_2)$ has at most two elements, they conclude
        that if $N$ has more than $2$ prime divisors then $E$ has even
        modular degree.
    \end{remark}

\subsection{Non-Semistable Case}
    \label{CMCase}
    The goal of this subsection is to prove the following
    \begin{thm}
        \label{InnerTwist}
        Let $A$ be an absolutely simple modular Abelian variety $A$ of level $N$ with
        an odd congruence number. Let $\delta_p=0$ for the odd primes $p$ and 
        $\delta_2=2$.
        Assume that $p^{2+\delta_p} | N$. Then $A$ has good reduction 
        away from $p$ and $2$, and has potentially good reduction everywhere. 
        Specifically, if $p$ is odd and $p^2|N$, 
        then $N=p^s$, $N=4p^s$, or
        $N=8p^s$ for $s \geq 2$, and if $16|N$ then $N=2^s$.
    \end{thm}
    We expect this theorem to be true without assuming $A$ to be absolutely
    simple; however, at this moment we do not know
    how to overcome the difficulty with the inner forms in that case. 
    To prove this theorem, we use the technique of Calegari and Emerton to
    show that such modular Abelian varieties have inner twists or complex
    multiplication
    by a character of conductor $p$ (see \cite{CE}). 
    Using the results of Ribet on inner
    twists \cite{RibetTwists}, we will prove that $A$ must have potentially good reduction
    everywhere if $A$ is absolutely simple, and that $A$ has good reduction away from
    $p$, and possibly $2$. 
    We have the following lemma.
    \begin{lemma}
        If $\End_{\QQbar}(A)\otimes \QQ$ is a matrix algebra, then
        $A$ is not absolutely simple.
        \label{TateLemma}
    \end{lemma}
    \begin{proof}
        Assume that $R=\End_{\QQbar}(A)\otimes \QQ$ is a matrix algebra.
        We can find the projections $e_1, e_2 \in R$ such that $e_1+e_2=\id$,
        $e_1 e_2=0$, and $e_1, e_2 \not \in \{ 0, \id\}.$
        For some integer $n$, 
        $ne_i \in \End_{\QQbar}(A)$. If we assume that $A$ is absolutely
        simple, the image of $ne_iA$ must be $A$ or $0$. However, since
        $(ne_1)(ne_2)=n^2e_1e_2=0$, one of them must be $0$. Assume 
        without loss of generality that
        $ne_2=0$ in $\End_{\QQbar}(A)$. This implies that $e_2=0$,
        which contradicts our assumption that $e_2 \not \in \{0, \id\}.$
        Therefore, $A$ is not absolutely simple.
    \end{proof}

    This lemma is used in conjunction with Ribet's result on the endomorphism algebra
    of modular Abelian varieties with inner twists. Specifically, let $A$ be a 
    $d$-dimensional simple modular
    Abelian variety. There are $d$ modular eigenforms of weight $2$
    and level $N$ associated with
    $A$, which are Galois conjugate to each other.
    Let $f=\sum a_n q^n$ be one such eigenform, and
    $E=\QQ(\dots,a_n,\dots)$ be the field of
    definition of $f$. We know that 
    $\End_{\QQ}(A) \otimes \QQ=E$. Let $D=\End_{\QQbar}(A)\otimes \QQ$ be the algebra
    of all of the endomorphisms of $A$. From \cite{RibetEndo} we know that
    $E$ is its own commutant in $D$,
    and therefore $D$ is a central simple algebra over some subfield 
    $F$ of $E$.
    If we assume that $A$ is absolutely simple, then $D$ must be some 
    division algebra with centre $E$.
    Furthermore, $D$ must be either $E$ (which forces $E=F$) or a 
    quaternion division algebra over $F$ (which forces $E$ to be a quadratic 
    extension of $F$).

    \begin{prop}
        Let $A$ be an absolutely simple modular Abelian variety $A$ of level $N$
        with an odd congruence number. Let $\delta_p=0$ for odd primes and $\delta_2=2$.
        If $p^{2+\delta_p} | N$  
        then $A$ has potentially good reduction everywhere, specifically, for 
        any other prime number $q$ if $q | N$ then $q^2 |N$.
    \end{prop}
    \begin{proof}
        Assume that $A$ is of dimension $d$, and let $f_A=\sum a_nq^n \in \CC((q))$
        be a normalized eigenform associated with $A$. Let $E=\QQ(\dots,a_i,\dots) \subset \CC.$ 
        Let $\chi$ be the quadratic character with conductor $p$.
        Since $p^{2+\delta_p} | N$, $\chi \otimes f_A$ is 
        another modular eigenform
        in $S_2(\Gamma_0(N))$ (see \cite{Shim94}). Since $\chi$ is a quadratic character, $\chi$ takes 
        values in $\pm 1$, and as a result $\chi \otimes f_A \equiv f_A \pmod \lambda$ for any 
        $\lambda | 2$. 
        If $A$ has an odd congruence number, then $\chi \otimes f_A$ must be in the same conjugacy 
        class as $f_A$. If $\chi \otimes f_A = f_A$, then $A$ has complex multiplication
        by $\chi$, and therefore $A$ has potentially good reduction everywhere.
        In this case, $A$ must be an elliptic curve, because
        if $A$ has complex multiplication and has a dimension greater than $1$,
        then the ring of endomorphisms of $A$ is a matrix algebra, which 
        contradicts the absolute simplicity assumption.
        In general, $A$ might have an inner twist, and $\chi \otimes f_A = \gamma(f_A)$ for some 
        $\gamma \in \hom(E,\CC)$. Let $\Gamma \subset \hom(E,\CC)$ such that for any
        $\gamma \in \Gamma$ we can find a character $\chi_{\gamma}$ such that 
        $\chi_\gamma \otimes f_A = \gamma(f_A)$. By \cite{RibetTwists},
        $F=E^\Gamma$
        and (as discussed above) $D=\End_{\QQbar}A\otimes \QQ$ must be a quaternion algebra.
        However, using theorem 3 of \cite{RibetEndo}, $A$ has potentially
        good reduction everywhere, as desired.

        The final claim of the lemma follows by noting that if $q |N$ but $q^2 \ndiv N$, then 
        $A$ has multiplicative reduction over any field extension.
    \end{proof}

    We now proceed to prove theorem \ref{InnerTwist}.
    Assume that $p^{2+\delta_p} |N$ and $q^{2+\delta_q} | N$ for 
    distinct primes $p$ and $q$.
    In this case, assuming that $A$ has no complex multiplication,
    $A$ has more inner twists, and the subset $\Gamma \subset \hom(E,\CC)$ will have at
    least four elements, $\gamma_1, \gamma_p, \gamma_q,$ and $\gamma_{pq}$. But that means
    that $|E:F|\geq 4$, which shows that $D$ must be a matrix algebra. However, lemma \ref{TateLemma}
    forces $A$ not to be absolutely simple, which contradicts our assumption.
    Since we are assuming $A$ is absolutely simple
    if $A$ has complex multiplication, then $A$ is an elliptic curve.
    Therefore it will have complex multiplication 
    by $\chi_p$ and $\chi_q$, which is impossible.
    This completes the proof of the main theorem in this section.

\subsection{Algebraic Congruence Number}
\label{sec23}
    In this section we show that a modular Abelian variety with odd congruence number 
    has bad reduction at no more than two primes. Let $A$ be an absolutely 
    simple optimal Abelian variety of conductor $N$. Let $B=\ker(\phi)$ where
    $\phi$ is the modular uniformization map $\phi:J_0(N) \rightarrow A$.
    Assume that $N$ is a not a power of a prime.
    Then theorem \ref{thmCE1} says that $A[2](\QQ)$ has a non-trivial
    element. Let $z \in A[2](\QQ)$ be a nontrivial rational $2$-torsion 
    point of $A$, and let $\mm \subset \TT$ be the annihilator of $z$.
    Since $z \in A[\mm] \neq 0$, we get that $A^\dual[\mm] \neq 0$. 
    Therefore, if 
    $B[\mm] \neq 0$ as well, then $A$ will have an even congruence number.
    We will show that when $N$ has more than two prime divisors, then $B[\mm] \neq 0$.

    We have the following lemma.
    \begin{lemma}
        Let $A$ be a new simple modular Abelian variety, $0 \neq z=A[2](\QQ)$, and
        let $\mm$ be the annihilator of $z$ in $\TT$.
        Then $\mm$ is generated by $2$, $T_l-(l+1)$ for $l \ndiv N$,
        $T_p-1$ for $p |N$ but $p^2 \ndiv N$, and $T_p$ for $p^2 | N$.
        \label{LemRib2}
    \end{lemma}
    \begin{proof}
        Clearly $z$ is killed by $2$, and by the Eichler-Shimura relationship,
        $T_l(z)=(\frob_l+l/\frob_l)(z)=(l+1)z$, since $z$ is rational.
        Since $A$ is a new modular Abelian variety, if $p || N$, 
        we have $T_p(z)=\pm z=z$, and if $p^2 |N$ then $T_p(z)=0$.
        This is the desired the result.
    \end{proof}

    Recall that $C_N \subset J_0(N)$ is the rational cuspidal subgroup of
    $J_0(N)$. Let $\mm \subset \TT$ be the annihilator of $z \in A[2]$.
    By definition we have that if $B[\mm] \neq 0$, then $A$ will have
    an even congruence number.
    We can use proposition \ref{newness} to show that $B[\mm] \neq 0$ when 
    $N$ has more than two prime divisors. Specifically, if $v \in C_N$
    of even order such that $\phi(v)=0$, then $v \in B \cap C_N$. 
    Now if $v$ is a cusp of the type considered in proposition \ref{newness}
    and of even order, then for some integer $\lambda$ we have
    that $\lambda v \in C_N[\mm]$.
    Therefore, we only need to check that such $v$'s have even order
    and that $\phi(v)=0$ to show that $A$ has an even congruence number.
    \begin{thm}
        Let $A$ be a new absolutely simple optimal modular Abelian variety with an odd congruence number.
        Then $N$ has at most two prime factors.
        \label{CongNumberPQ}
    \end{thm}
    \begin{proof}
        If $A$ has an inner twist or complex multiplication, 
        then the result follows by theorem \ref{InnerTwist}.
        Assume that $A$ has an odd congruence number with no inner twist or
        complex multiplication.
        Assume to the contrary that $N$ has more than two prime
        factors. Then $N=2^\alpha M$ with $M$ square free odd integer, and
        $\alpha<4$. Furthermore, by theorem \ref{thmCE1}, we can find a nontrivial
        $z \in A[2](\QQ)$. Let $\mm$ be the annihilator of $z$.
        We now find $v \in C_N$ of the form considered in proposition
        \ref{newness} such that $v$ has even order and $\phi(v)=0$.
        We will consider three main cases, based on the valuation of 
        $N$ at $2$.
     
        Assume that $4 \ndiv N$.
        Since $w_N=\prod_{l|N}w_l$, and $(w_N)_*=-1$,
        there is an odd number of primes such that $(w_l)_*$
        act as $-1$ on $A$. Therefore, we can select three distinct prime divisors
        of $N$, call them $p$, $q$, and $r$, 
        such that $(w_p)_*$ acts as $-1$, 
        while $(w_r)_*=(w_q)_*$. 
        If $2 || N$, by lemma \ref{AtkinLehnerSign} 
        $(w_2)_*$ acts as $+1$. Therefore, without loss of generality
        assume that $2 \ndiv pq$.

        Let $s_p$, $s_q =\pm 1$ and let 
        \[ v=(1-w_{qr})(1+s_p w_p)(1+s_q w_q) P_1 = (1+s_p w_p)(1+s_q w_q)(1-s_q w_r)P_1.\]
        By the computation from section \ref{Cuspidal} we get that
        $v$ has order
        $\num\left( {(1+s_p p)(1+s_q q)(1-s_q r) \over 24} \right)$.
        If we select $s_p \equiv -p \pmod 4$ and $s_q \equiv -q \pmod 4$, then
        this order is even. Furthermore, note that $v$ is of the form
        considered in proposition \ref{newness}, so we only need to show
        that $\phi(v)=0$ to prove $A$ has an even congruence number.
        Note that $\pi(w_{qr}(\tau))=\pi(\tau)+a$ 
        for any $\tau \in X_0(N)$, where $a$ is some $2$-torsion point. 
        Let $P=(1+s_pw_p)(1+s_qw_q)P_1=P_1\pm P_p \pm P_q \pm P_{pq}$.
        Then 
        \begin{eqnarray*}
            \phi(v) &=& \phi(w_{qr}(P)-P) \\
            & = & \sum_{m|pq} \pi(w_{qr}(P_m))-\pi(P_m) \\
            &=& 4a =0,
        \end{eqnarray*}
        which shows that $A$ has an even congruence number.

        Assume that $4 || N$. By lemma \ref{AtkinLehnerSign} we 
        know that $(w_2)_*$ acts as $-1$.
        Let $p,q | N$ and let $v=(1-w_p)(1+s_q w_q)P_2$ with $s_q=\pm 1$. 
        The order of $v$ is 
        $\num\left( {(1-p)(1+s_q q) \over 4}\right)$.
        If we select $s_q \equiv -q \pmod 4$, then $v$ will have an even
        order. Again note that $v$ is of the form considered in proposition
        \ref{newness}.
        Since $(w_2)_*$ is acting as $-1$, either $(w_p)_*$ or
        $(w_{2p})_*$ is acting trivially on $A$. Let $w$ be the corresponding
        Atkin-Lehner involution. Note that because $w_2(P_2)=P_2$,
        $v=(1-w)(1+s_q w_q)P_2$.
        Furthermore,
        $\pi(w(\tau))-\pi(\tau)=a\in A[2]$ for any $\tau \in X_0(N)$.
        As a result
        \[ \phi(v)=\pi(P_2)-\pi(w(P_2))+s_q(\pi(P_{2q})-\pi(w(P_{2q})) = a+s_qa=0.\]
        Therefore $\phi(v)=0,$ which proves that in this case
        $A$ has an even congruence number.

        Finally assume that $8 || N$, and let $p, q |N$ be two distinct odd
        divisors of $N$. Let $(w_p)_*$ and $(w_q)_*$
        act as $s_p$ and $s_q$ on $A$. Let
        \[ v=(1-w_2)(1+s_pw_p)(1+s_qw_q)P_1 = (1-w_2)(1+s_ps_q w_{pq})(1+s_pw_p)P_1.\]
        Then $v$ has order
        $\num\left( {(1+s_pp)(1+s_qq) \over 2}\right)$ that is even. 
        Again, $v$ is of the form considered in proposition \ref{newness},
        and similar to the case when $N$ is odd, we can write
        $v=(1-w)P$ for some Atkin-Lehner involution $w$ such that
        $w_*=1$ and some $P=(1-w_2)(1\pm w')$. That shows
        $\phi(v)=0$. Therefore $A$ in this case will have an even congruence
        number again.
    \end{proof}
    Combining this result with the main result of section \ref{CMCase}, we get 
    \begin{cor}
        Let $A$ be an absolutely simple modular Abelian variety with an odd congruence
        number and conductor $N$. Then $N$ has at most two prime divisors.
        Furthermore, if $N$ is not square free, then $N=2^a$, $p^b$, $4p^b$ or $8p^b$, where
        $p$ is an odd prime.
    \end{cor}

\subsection{Congruence Classes of Primes}
    Let $A$ be a simple modular Abelian variety of conductor $N$ 
    with an odd congruence number, and
    without complex multiplication or an inner twist. 
    As usual let $\pi:X_0(N)\rightarrow A$ to be the composition of
    the Albanese embedding with the modular uniformization $\phi$.
    Assume that $N$ is not a power of a prime, which by theorem 
    \ref{thmCE1} implies that $A[2](\QQ)$ is nontrivial.
    From the previous sections we know that 
    $N$ has at most two prime factors, say $p$ and $q$. 
    In this section we find congruences that $p$ and $q$ must satisfy. 
    As in the proof of theorem \ref{CongNumberPQ}, we use different
    techniques depending on the valuation of $N$ at $2$.
    
    If $N$ is odd, then $N=pq$ with both $p$ and $q$ being odd.
    By lemma \ref{AtkinLehnerSign}, we know that $(w_{pq})_*$ is acting as $-1$ 
    on $A$.
    Therefore, assume without loss of generality that $(w_q)_*$ is acting
    trivially on $A$ and $(w_p)_*$ is acting as $-1$. 
    Let $v=(1\pm w_p)(1-w_q)P_1$. Again, $\pi(\tau)-\pi(w_q(\tau))=a \in A[2]$
    for all $\tau \in X_0(N)$.
    As a result,
    \[ \phi(v)=\pi(P_1)-\pi(w_q(P_1))\pm(\pi(P_p)-\pi(w_q(P_p)))=a\pm a=0.\]
    Note that the order of $v$ is
    $\num\left( {(p\pm 1)(q- 1) \over 24} \right).$
    Since we are assuming that $A$ has odd congruence number, we get that
    $p \equiv \pm 3 \pmod 8$ and $q \equiv 3 \pmod 4$.

    We record a useful corollary of the above result.
    \begin{cor}
        Let $A$ be a modular Abelian variety with conductor $pq$, $p$ and $q$
        both odd, and an odd
        congruence number. Then $A[2](\QQ)$ is at least $2$-dimensional
        over $\FF_2$.
        \label{twotorsionstructure}
    \end{cor}
    \begin{proof}
        We prove this by finding two distinct points in $C_N[\mm].$
        First note that $P_1-P_p$ and $P_1-P_q$ have the orders
        $(p-1)(q^2-1)/24$ and $(p^2-1)(q-1)/24$, respectively. Therefore,
        both
        \begin{eqnarray*}
            u={(p-1)(q^2-1) \over 48}(P_1-P_p), u'={(p^2-1)(q-1)\over 48}(P_1-P_q)
        \end{eqnarray*}
        are of order $2$.
        We can easily check that $T_pu=u$ and $T_q u'=u'$.
        On the other hand
        \[u+T_q u={(p-1)(q^2-1)\over 48}(P_1-P_p+P_q-P_{pq}),\]
        which is zero. Similarly, we get $u'+T_pu'=0$. Therefore,
        $u, u' \in C_N[\mm]$. 
        Furthermore, we know that
        $\Lambda(u+u')$ has integral coefficients, but
        \[(1,0)\otimes(1,1) \Lambda(u+u')=(q-1)/2,\]
        which is not even since $q \equiv 3 \pmod 4.$
        Therefore, $u+u' \neq 0$, which implies that
        $C_N[\mm]$ is at least $2$-dimensional over $\FF_2$.
        Since we are assuming that
        $A$ has an odd congruence number, $C_N[\mm]$ injects in
        $A$, which is the desired result.
    \end{proof}

    If $N=2p$, we know by lemma \ref{AtkinLehnerSign} 
    that $(w_2)_*$ acts trivially and $(w_p)_*$ acts as $-1$ on $A$.
    Therefore, $\pi(P_2)=\pi(w_2(P_1)) \in A[2]$, and $P_2-P_1$
    (which has order ${p^2-1 \over 8}$) must have an even order. Let 
    $v={p^2-1 \over 16}(P_2-P_1) \in C_N[2]$. By proposition 
    \ref{newness}, $T_p(v)=T_2(v)=v$, hence $v \in C_N[\mm]$.
    Note that 
    \[ \phi(v)=\pi\left({p^2-1 \over 16}(P_2-P_1)\right)={p^2-1 \over 16}\pi(P_2),\]
    so if ${p^2-1 \over 16}$ is even, then $\pi(v)=0$. This
    implies that $z \in C_N[\mm] \cap B$, and, in turn, that
    the congruence number is even. Since we are assuming that the congruence
    number of $A$ is odd, we get that ${p^2-1 \over 16}$ is odd, 
    that is $p^2-1 \equiv 16 \pmod {32}$.
    That implies that $p \equiv \pm 7 \pmod {16}$.
    However, we also know that $w_2$ cannot have any fixed points.
    This implies that $-2$ is not a quadratic residue mod $p$, which 
    means that $p \equiv 5$, $7$, $13$, or $15 \pmod {16}$.
    Therefore $p \equiv 7 \pmod {16}$.

    If $N=4p$, then we know that $(w_2)_*$ acts as $-1$ on $A$, while $(w_p)_*$
    acts trivially. Therefore, $\pi(P_2)-\pi(P_{2p})=\pi(P_2)-\pi(w_p(P_2)) \in A[2]$. 
    The
    order of $P_2-P_{2p}$ is ${p-1 \over 2}$. Therefore, if $A$ has 
    an odd congruence number, 
    $(p-1)/4$ must be odd, hence $p \equiv 5 \pmod 8$.

    If $N=8p$, we can check that $(1-w_2)(1-w_p)P_1$ vanishes in $A$,
    and that it has order $p-1 \over 2$. Therefore, $4 \ndiv p-1$, otherwise
    $A$ will have an even congruence number. Therefore $p \equiv 3 \pmod 4.$
    (We can probably say more, if we figure out the sign of $(w_p)_*$.)

    We combine the above results in the following theorem.
    \begin{thm}
        \label{modabclassification}
        Let $A$ be a new modular Abelian variety with an odd congruence number
        and conductor $N$. Assume that $A$ has no inner twists or complex
        multiplications.
        Then one of the following must be true
        \begin{enumerate}
            \item $N$ is a prime number $p$.
            \item $N=pq$ and $p \equiv \pm 3 \pmod 8$ and $q \equiv 3 \pmod 4$.
            \item $N=2p$ and $p \equiv 7 \pmod {16}$.
            \item $N=4p$ and $p \equiv 5 \pmod 8$.
            \item $N=8p$ and $p \equiv 3 \pmod 4.$
        \end{enumerate}
    \end{thm}

    \section{Elliptic Curves with Odd Congruence Numbers}
\label{EllipticCurves}
In this section, we apply the results of the previous section to the
case of elliptic curves. We show that the conductors of all such
elliptic curves are of the form $p$, $pq$, $2p$, $4p$, or one of the finitely
many exceptions. We study each class to demonstrate that all such elliptic curves
have finite Mordell-Weil group, except possibly when the conductor is prime. Furthermore,
we know from the result of \cite{ARS} that when $4 \ndiv N$, then having an
odd congruence number is the same as having odd modular degree. We conjecture
that in fact having an odd congruence number is equivalent to odd modular 
degree in all cases. As a result, we can
state many of our results in terms of modular degrees.

\subsection{Complex Multiplication} 
\label{ComplexMultiplication}
  Let $E$ be an elliptic curve of conductor $N$. 
  If $p^2 | N$ for an odd prime $p$, then by section \ref{CMCase} we know that
  $E$ has complex multiplication. We also showed that if $16 | N$ then 
  $E$ must have complex multiplication.
  There are only finitely many elliptic curves over rationals 
  with complex multiplication 
  and the conductor $2^mp^n$ for some prime number $p$. The following is
  the list of all such elliptic curves that have an odd modular degree: 
  $E=27A, 32A, 36A, 49A, 243B$.
  We also verify that all such elliptic curves have rank $0$, as predicted
  by Watkins's conjecture.

  We will now focus our attention on elliptic curves without complex 
  multiplication, that is elliptic curves with conductor $N=p,2p,4p,8p,$ or
  $pq$ for some odd primes $p$ and $q$.
  Each of the remaining sections deals with one of these remaining cases.
\subsection{Prime Level}
\label{PrimeLevel}
  Let $E$ be an elliptic curve with an odd congruence number and a prime
  conductor $N$. 
  Mestre and Oesterl{\`e}
  \cite{MO89} have studied elliptic curves of prime conductors, and they have
  demonstrated
  that aside from elliptic curves $11A$, $17A$, $19A$, and $37B$, all such elliptic
  curves have either a trivial torsion subgroup or a $\ZZ/2\ZZ$ torsion subgroup. 
  The above cases have the
  torsion structures $\ZZ/5\ZZ$, $\ZZ/4\ZZ$, $\ZZ/3\ZZ$, and $\ZZ/3\ZZ$, respectively.
  Mestre and Oesterl{\`e} also showed that if $E_\tors$ is $\ZZ/2\ZZ$, 
  then $E$ is a Neumann-Setzer curve
  and $N=u^2+64$. Stein and Watkins have studied the parity of congruence
  number of Neumann-Setzer curves (see \cite{SW2004}) and they show that 
  $E$ has odd congruence number if and only if $u \equiv 3 \pmod 8$.
  Furthermore one can show that Neumann-Setzer curves have rank $0$
  using descent. We will give another proof of this fact using $L$-functions.
  \begin{prop}
	  Let $E$ be an elliptic curve over $\QQ$ with a prime conductor $N$. Assume that
	  $E_\tors$ is nontrivial. Then $L(E,1) \neq 0$, hence $E(\QQ)$ has rank
	  $0$.
  \end{prop}
  \begin{proof}
      Recall that 
      \[L(E,1)=2\pi i \int_0^{i \infty} f_E(z) dz \equiv \pi(P_N) \pmod {\Lambda_E},\]
      where $\CC/\Lambda_E \isom E(\CC)$.
      Therefore, if $L(E,1)=0$, then $\pi(P_N)=0$, or alternatively $\phi(P_1-P_N)=0$.
      By \cite{Maz79} and \cite{MO89} (see also \cite{Eme03})
      we know that $J_0(N)_\tors$ is generated by the cusp $P_1-P_N$,
      and for any elliptic curve quotient of 
      $J_0(N) \rightarrow E$, $E_\tors$ is generated
      by the image of $\pi(P_1)-\pi(P_N)$. Since we are assuming that $E$ has 
      a nontrivial torsion structure, $\pi(P_1)-\pi(P_N)\neq 0$, 
      which implies that $L(E,1) \neq 0$.
      Therefore rank of $E(\QQ)$ is zero by work of \cite{Kol} and \cite{GrZa}.
  \end{proof}

  The case when $E$ has a trivial torsion structure and 
  an odd congruence number is studied by 
  Calegari and Emerton (see \cite{CE}), where they show that 
  $E$ has an even analytic rank (since $(w_N)_*=-1$), 
  supersingular reductions at $2$ and $E(\RR)$ is connected.
  Doing a search in the Cremona's database, it appears that
  if an elliptic curve $E$ has supersingular reduction at $2$, 
  Mordell-Weil rank $0$, connected real component, then $E$
  will have an odd congruence number.
\subsection{Level $N=pq$}
\label{twotozero}
  In this subsection, we will study elliptic curves of odd modular degree
  and conductor $N=pq$ where $p$ and $q$ are both odd primes. Let $E$
  be such an elliptic curve.
  Assume throughout this section that $(w_p)_*=-1$ on $E$.
  By theorem \ref{modabclassification}, we know that $p \equiv \pm 3 \pmod 8$
  and $q \equiv 3 \pmod 4$. We will show that with a few exceptions, 
  $p,q \equiv 3 \pmod 8$, and that 
  all such elliptic curves have finite Mordell-Weil group over $\QQ$.

  Recall that by corollary \ref{twotorsionstructure} we know that 
  $E[2](\QQ)=(\ZZ/2\ZZ)^2$.
  First, we show that if $E_{\tors}$ is $\ZZ/2 \times \ZZ/4$, then 
  $E$ has conductor $15$ or $21$. We can prove a general result
  about semistable elliptic curves with $E_{\tors}=\ZZ/2 \times \ZZ/4$ and 
  good reduction at $2$.
  Specifically
  \begin{lemma}
      Let $E$ be a semistable elliptic curve with good reduction at $2$.
      $E_{\tors}=\ZZ/2 \times \ZZ/4$, and let $Q \in E(\ZZ[1/N])$ be a point
      of order $4$. Let $\overline{Q}$ be the reduction of $Q$ modulo $2$.
      Then $\overline{Q}$ has order $4$ in $E(\FF_2)$.
      \label{technical3}
  \end{lemma}
  \begin{proof}
    We can check that an elliptic curve $E$ with good reduction at $2$ and
    a rational $2$-torsion point has a minimal model 
    \[ E: y^2+xy=x^3+a_2x^2+a_4x.\]
    Since $E[2] =\ZZ/2 \times \ZZ/2$,
    $(4a_2+1)^2-64a_4$ is a perfect square.
    The $x$ coordinates of the $2$-torsion points are $0$, $4\alpha$, and 
    ${\beta \over 4},$ were
    $\alpha$ and $\beta$ are 
    both (odd) integers since we are assuming that $E$ is in minimal model. 
    Furthermore, since $E$ is assumed to be semistable,
    $\alpha$ and $\beta$ are coprime to each other.
    Note that the point $({\beta \over 4},-{\beta\over 8}) \in E(\QQ)$ maps to 
    the identity under the reduction mod $2$ map.
    Using the notation from \cite{AECI}, we have 
    \begin{eqnarray*}
        b_2&=& 16 \alpha+\beta, \\
        b_4&=& 2\alpha \beta, \\
        b_6&=& 0, \\
        b_8&=& -\alpha^2 \beta^2,\\
        \Delta&=& \alpha^2 \beta^2 (16\alpha-\beta)^2.
    \end{eqnarray*}
    Let $Q \in E(\QQ)$ be a point of order $4$, and let $x(Q)=x_0$.
    Recall that we want to show $\overline{Q} \in E(\FF_2)$ is a point of
    order $4$.
    We have that $x([2]Q)=0$, $4\alpha$, or ${\beta \over 4}.$
    If $\overline{Q}$ has order less than $4$, then $2\overline{Q}$ must be
    the identity element, that implies that $x([2]Q)={\beta \over 4}$.
    In that case
    \begin{eqnarray*}
        {\beta \over 4} &=& {x_0^4-b_4x_0^2-b_8 \over 4x_0^3+b_2x_0^2+2b_4x_0} \\
        &=& {x_0^4-2\alpha \beta x_0^2 + \alpha^2 \beta^2 \over 4x_0^3+(16 \alpha+\beta)x_0^2+4\alpha \beta x_0}, \\
        \Rightarrow 0 &=&  x_0^4-\beta x_0^3-(6\alpha \beta+{\beta^2 \over 4})x_0^2-\alpha \beta^2 x_0 + \alpha^2 \beta ^2 \\
        &=& (x_0^2-{\beta \over 2} x_0 + \alpha \beta)^2 - (4\alpha \beta+{\beta^2 \over 2})x_0^2.
    \end{eqnarray*}
    Therefore, $16\alpha \beta+2\beta^2=2\beta(8\alpha+\beta)$ must be 
    a perfect square; however that is
    not possible because $\alpha$ and $\beta$ are odd. As a result 
    $x([2]Q)=0$ or $4 \alpha$. Therefore, $[2]\overline{Q}$ has order $2$
    in $E(\FF_2)$. This shows that $\overline{Q}$ has order $4$, which is the desired result.
  \end{proof}

  \begin{prop}
	Let $E$ be an elliptic curve with conductor $pq$ and $E_{\tors}=\ZZ/2 \times \ZZ/4$.
	Then, $pq=15$ or $21$.
	\label{Technical2}
  \end{prop}
  \begin{proof}
    Using the same notation as in lemma \ref{technical3},
    let $0$, $4\alpha$ and $\beta\over 4$ be the $x$-coordinates of the
    $2$-torsion points of $E$.
    Let $Q$ be a point in $E_{\tors}$ of order $4$.
    By lemma \ref{technical3}, $x([2]Q)=0$ or $4\alpha$.
    Without loss of generality, assume that $x([2]Q)=0$, since
	if $x([2]Q)=4\alpha$, then we can change the coordinates to find another model
	with $x([2]Q')=0$. 
    Let $x_0=x(Q)$.
    Then $x_0^4-2\alpha \beta x_0^2+\alpha^2\beta^2=0$,
    which implies that $x_0^2=\alpha\beta$. Since $\alpha$ and $\beta$ are coprime,
    they are both perfect squares, or negative of perfect squares (both of
    the same sign).
    Since $E$ is of conductor $pq$, 
    $\Delta=\alpha^2 \beta^2(16 \alpha-\beta)^2$ is a product of the powers of $p$ and $q$.
    Let $a^2=\pm \alpha$ and $b^2=\pm \beta$. 
    Then,
    $a^4 b^4(4a-b)(4a+b)$ is a product of the powers of $p$ and $q$. Note that 
    $(4a-b,4a+b)=1$, which implies that all factors are pairwise coprime.
    Note that if 
    $|4a+b|=|4a-b|=1$, then
    either $a=0$ or $b=0$ contrary to our assumptions. Therefore we will
    assume without loss of generality that $4a+b>1$.

    If $4a-b \neq \pm 1$, then $a^2=b^2=1$, which means $E$ is the
    elliptic curve $15A$.
    If $4a-b=\pm 1$ then $|b|>1$, therefore $|a|=1$. Since we are assuming
    that $4a+b>1$ we get that $a=1$, and $4a-b=1$ leads to elliptic curve
    $21A$ and $4a-b=-1$ leads to elliptic curve $15A$.
	This completes our proof.
  \end{proof}
  \begin{remark}
    Note that the previous proposition seems a bit tedious.
    It is straightforward to show that $3$ must divide the conductor by 
    the Hasse-Weil
	bound. 
    Unfortunately, it is not clear how this observation can simplify the argument.
  \end{remark}
  An immediate corollary of the above is that for an elliptic curve $E$
  of conductor $pq$ and ordinary reduction at $2$, we have $E_\tors=(\ZZ/2\ZZ)^2$, since the only 
  other option is $E_\tors=\ZZ/2\ZZ\times \ZZ/6\ZZ$. However the Hasse-Weil
  bounds for elliptic curves rules this case out.

  \begin{thm}
	Assume that $E$ is an elliptic curve with an odd modular degree. Furthermore, assume
	that the conductor of $E$ is $pq$ with $pq\neq 21$ or $15$. Then $p, q \equiv 3 \pmod 8$.
  \end{thm}
  \begin{proof}
    Note that by corollary \ref{twotorsionstructure} we know that 
    $E[2](\FF_2)$ is non-trivial, hence $E$ has good ordinary reduction
    at $2$.
    Therefore, for $pq \neq 21$ and $15$ we have $E(\QQ)_\tors=(\ZZ/2\ZZ)^2.$
    Recall that we are assuming $(w_p)_*=-1$ and $(w_q)_*=1$ on $E$.
    Note that
    \begin{eqnarray*}
        \pi(\tau)-\pi(w_p(\tau))&=&\phi(\tau-P_1) -\phi(w_p(P_1)-P_1) -\phi(w_p(\tau)-w_p(P_1)) \\
        &=& \pi(\tau)-(w_p)_*(\pi(\tau))-\pi(w_p(P_1)) \\
        &=& 2\pi(\tau)-\pi(P_p),
    \end{eqnarray*}
    for any $\tau \in X_0(N)$.
    When $\tau$ is a cusp of $X_0(N)$, $\pi(\tau)$ is a torsion point, 
    and since $E_{\tors}=E[2]$ we get $2\pi(\tau)=0$. Therefore
    \[ \pi(\tau)-\pi(w_p(\tau))=\pi(P_p).\]
    Let $v=(1+w_q)(1-w_p)P_1$. 
    Then
    \[\phi(z)=(\pi(P_1)-\pi(w_p(P_1)))+(\pi(P_q)-\pi(w_q(P_q))) = 2\pi(P_p)=0.\]
    As a result $v \in B \cap C_N$.
    Also, since $v$ is of the form that is considered in proposition
    \ref{newness}, if $v$ has even order then $E$ will have even congruence 
    number. Since we are assuming that $E$ has an odd congruence number, 
    $v$ must have an odd order.
    The order of this point
    is $\num((q+1)(p-1)/24).$ Since $q \equiv 3 \pmod 4$,
    $4 | q+1$. If $p \equiv -3 \pmod 8$, then
    $v$ will have an even order, and $E$ will have an even congruence number. 
    Therefore $p \equiv 3 \pmod 8$, 
    and $2 || p+1$. If $q \equiv -1 \pmod 8$, again $v$ will have
    an even order. Therefore, $q \equiv 3 \pmod 8$, which is the desired result.
  \end{proof}
  We also get the following corollary.
  \begin{cor}
    Assume that $E$ is an elliptic curve with an odd congruence number
    and the conductor $pq$
    with $pq \neq 15$ or $21$. Then there exist odd integers $r$ and $s$ such
    that $|p^r-q^s|=16$.
  \end{cor}
  \begin{proof}
    \label{diffsixteen}
    Following the notation of lemma \ref{technical3}, we have
    $\Delta=\alpha^2 \beta^2(16 \alpha-\beta)^2$ for some odd integers
    $\alpha$ and $\beta$, coprime to each other. Assume that $\alpha^2 \neq 1$,
    then $|\alpha|=p^r$, $q^s$, or $p^rq^s$. In the last case,
    $\beta^2=(16\alpha-\beta)^2=1$, which is not possible. Therefore 
    assume without
    loss of generality that $\alpha=\pm p^r$. If $\beta=\pm q^s$,
    $16 \alpha-\beta = \pm 1$, which leads to the Diophantine equation
    $\pm 16 p^r- \pm q^s = \pm 1$. We get the same Diophantine equation
    if $\beta=\pm 1$. Therefore, we
    need to solve the Diophantine equation
    \[ q^s-16 p^r=\pm 1.\]
    Since $q^s \equiv 3 \pmod 8$ for all odd $s$'s, and $q^s \equiv 1 \pmod {16}$
    for all even $s$'s, $s$ must be even and
    \[ q^s-16 p^r =1. \]
    This leads to $(q^{s/2}-1)(q^{s/2}+1)=16p^r$, and since $(q^{s/2}-1,q^{s/2}+1)=2$,
    $q^{s/2}=7$ or $9$. Therefore $q^s=81$, which forces $p=5.$ This is not
    congruent to $3 \pmod 8$, so we get that $\alpha = \pm 1$.

    If $\beta^2=1$, then $|\pm 16-\beta|$ is $15$ or $17$, which again contradicts 
    $p, q \equiv 3 \pmod 8$. We get the same result if $(\pm 16-\beta)^2=1$. Therefore,
    $\beta=\pm p^r$ and $\pm 16-\beta=\pm q^s$. This leads to the Diophantine equation
    $|p^r-q^s|=16$.
    Since $p,q \equiv 3 \pmod{8}$, $r \equiv s \pmod 2$. If they are both even,
    then the difference of the two squares equals $16$, which forces $N=15$.
    Therefore, $r$ and $s$ are odd, which is the desired result.
    Finally note that in this case the elliptic curve has the model
    \[ E: y^2+xy=x^3+{15+ p^r \over 4}x^2 +p^r x. \]
  \end{proof}

  We also have the following
  \begin{thm}
      Let $E$ be an elliptic curve with conductor $pq$ and an
      odd congruence number. Then $L(E,1) \neq 0$, hence $E$ has rank $0$.
  \end{thm}
  \begin{proof}
      For $pq=15$ or $21$ we can check that $E$ has Mordell-Weil rank $0$.
      Therefore assume that $pq \neq 15$ or $21$.
      Recall that in proposition \ref{twotorsionstructure} we showed that
      \[ u={(p-1)(q^2-1) \over 48}(P_1-P_p), u'={(q-1)(p^2-1) \over 48}(P_1-P_q),\]
      have order two, and $\phi(u)$ and $\phi(u')$ are linearly independent, hence
      they generate $E[2]$.
      However, since $p,q \equiv 3 \pmod 8$ we get that $u$ and $u'$ are odd
      multiples of $P_1-P_p$ and $P_1-P_q$, respectively. So $\pi(P_p)$ and $\pi(P_q)$
      also generate $E[2]$. Therefore, $\phi(P_p-P_q)$ is nontrivial. Applying the 
      Atkin-Lehner involution $w_p$ to $P_p-P_q$, 
      we get that $\phi(P_1-P_{pq})$ is nontrivial.
      Therefore, $\pi(P_{pq}) \neq 0$, which implies that $L(E,1) \neq 0$.
  \end{proof}

\subsection{Level $N=2p$}
\label{twotoone}
  In this section, we will study the case when $N=2p$ for $p$ an odd
  prime. Specifically, we want to show that $L(E,1) \neq 0$. 
  In this case it seems more straightforward to prove this using analytic tools. 

  Specifically, let $f_E(q)=\sum a_nq^n$ be the modular form attached to the 
  elliptic curve $E$, and let $\Omega_E$ be the real period of $E$.
  Note that $L(f_E, 1) \in \RR$ since the Fourier coefficients of $f_E$ are
  rational integers. Therefore, the order of $\pi(P_{2p})$ is the order of
  $L(f_E,1) \in \RR/{\Omega_E \ZZ}$.
  We know that $L(f_E,s)$ has an Euler product expansion
  \[ L(f_E,s)= \prod_p L_p(f_E,s),\]
  and $L_2(f_E,s)={1 \over 1-a_22^{-s}}.$
  Similarly 
  \begin{eqnarray*}
      \pi(P_{p})&=& 2\pi i\int_{1\over 2}^{i\infty} f_E(z)dz \\
      &=& 2\pi i \int_0^{i\infty} f_E(z+1/2)dz  \\
      &=& 2\pi i \int_0^{i\infty} \sum (-1)^{n}a_n q^ndz
  \end{eqnarray*}
  which implies that $\pi(P_p)$ can be written as $L(g,1)$ where $L(g,s)$
  has an Euler product expansion
  \begin{eqnarray*}
      L(g,s)&=&  ({-1+{a_2 \over 2^s}+{a_4\over 4^s}+\dots})\prod_{p>2} L_p(f_E,s) \\
      &=& -{1-a_2 2^{1-s} \over 1-a_2 2^{-s}} \prod_{p>2}L_p(f_E,s)
  \end{eqnarray*}
  Therefore $L(g,1)=L(f_E,1)(a_2-1),$
  and more appropriately for us 
  \[ \pi(P_p)\equiv (a_2-1)\pi(P_{2p})\pmod {\Omega_E \ZZ}.\]
  
  We know that if $E$ has an odd congruence number, then $(w_2)_*$ is acting
  trivially, which implies that $a_2=-1$. Therefore
  \[ \pi(P_p) \equiv -2\pi(P_{2p})  \pmod {\Omega_E \ZZ}.\]
  However, we also know that $P_{2p}=w_2(P_{p})$, and $\pi(w_2(P_{p}))=\pi(P_{p})+\alpha$
  where $\alpha$ is a $2$-torsion point in $E$. Since both
  $\pi(P_p)$ and $\pi(P_{2p})$ are equivalent to real numbers, 
  $\alpha$ is also equivalent to a real number, which 
  implies that $\alpha \equiv {\Omega_E \over 2} \pmod {\Omega_E \ZZ}.$
  As a result
  \begin{eqnarray*}
      \pi(P_p) &\equiv& \pi(P_{2p})+{\Omega_E \over 2} \pmod {\Omega_E \ZZ}, \\
      & \equiv & -2\pi(P_{2p}) \\
      \Rightarrow -3\pi(P_{2p}) &\equiv& {\Omega_E \over 2} \pmod {\Omega_E \ZZ}, \\
      \Rightarrow \pi(P_{2p}) &\equiv& \Omega_E( {k \over 3}-{1 \over 6}) \pmod {\Omega_E \ZZ}
  \end{eqnarray*}
  for some integer $k$. Therefore, $\pi(P_{2p}) \neq 0$ and 
  $L(f_E,1) \neq 0$. We also observe that $\pi(P_{2p})$ will either
  be a $6$-torsion point (for $k \equiv 0$ or $1 \pmod 3$), or a 
  $2$-torsion point (for $k \equiv 2 \pmod 3$).

  In either case, we have an elliptic curve with a conductor
  $2p$ and a rational $2$-torsion point. Such elliptic curves have
  been studied by Ivorra \cite{Ivorra}. We can use his techniques
  to put stringent conditions on the values for $p$.
  Ivorra shows that if $p \geq 29$, then 
  there is an integer $k \geq 4$ such that
  one of $p+2^k$, $p-2^k$, or $2^k-p$ is a perfect square.
  However, we already know from theorem \ref{modabclassification} that 
  $p \equiv 7 \pmod{16}$. Putting these
  two together, we get that $p=2^k-m^2$.
  In fact, in this case, Ivorra's result says that there exists
  $7 \leq k < f(p)$ where
  \begin{eqnarray*}
  f(n)= 
  \begin{cases}
      18+2\log_2 n & \mbox{ if $n<2^{96}$,} \\
      435+10\log_2 n & \mbox{ if $n\geq 2^{96}$}
  \end{cases},
  \end{eqnarray*}
  and our elliptic curve is isogeneous to
  \[ y^2+xy=x^3+{m-1 \over 4}x^2+2^{k-6}x.\]
  Searching through the Cremona database, we find out
  that the only elliptic curves with an odd modular degrees and conductors
  $2p$ with $p \leq 29$ are $E=14A$ and $E=46A$, and both of these
  are of the form above.

\subsection{Level $N=4p$}
\label{twototwo}
    As with the case of $N=2p$, we can use Ivorra's table to parametrize
    all elliptic curves with conductor $4p$ and a rational $2$-torsion point.
    Specifically, for $p>29$, 
    $p=a^2+4$ for some integer $a \equiv 1 \pmod 4$, and
    $E$ is isomorphic to one of the following two isogenous elliptic curves
    \begin{eqnarray*}
        E &:& y^2=x^3+ax^2-x, \\
        E'&:& y^2=x^3-2ax^2+px.
    \end{eqnarray*}
    We can calculate the rank of such elliptic curves using a standard
    $2$-descent. In fact, if we let $\phi:E \rightarrow E'$ and
    $\phi'$ be the dual isogeny, using the notation from \cite{AECI} we
    get 
    \[ |S^\phi(E,\QQ)|=|S^{\phi'}(E,\QQ)|=2, \]
    which implies that
    \[ |E(\QQ)/\phi'(E'(\QQ))|=|E'(\QQ)/\phi(E(\QQ))|=2, \]
    which, by the exact sequence
    \[ 0 \rightarrow E'(\QQ)[\phi']/\phi(E(\QQ))[2] \rightarrow E(\QQ)/\phi'(E'(\QQ)) \rightarrow E(\QQ)/2E(\QQ) \rightarrow E'(\QQ)/\phi(E(\QQ)) \rightarrow 0 \]
    gives us $|E(\QQ)/2E(\QQ)| \leq 4.$ This forces the rank
    of $E(\QQ)$ to be $0$. 
    
    For $p \leq 29$, we can consult Cremona's table to get the 
    elliptic curves $20A$, $52C$, and $116C$. In fact all these elliptic curves
    are of the model constructed above.

\subsection{Level $N=8p$}
\label{twotothree}
    In this case, Ivorra's table tells us that any elliptic curve with a rational
    $2$-torsion point and the conductor $N=8p$ satisfies 
    $p \equiv a^2 \pmod {16}$ for $p>31.$ However, by 
    theorem \ref{modabclassification}, $p \equiv 3 \pmod 4$,
    therefore there are no elliptic curves with conductor $8p$ and
    odd congruence number
    for $p >31.$ Using Cremona's table, we know that the elliptic curve
    $24A$ is the only elliptic curve with the conductor $8p$ and an odd
    congruence number.
    Furthermore this curve has rank $0$.

We will combine all of the above results in
\begin{thm}
    \label{oddellipticcurves}
    Let $E/\QQ$ be an elliptic curve with an odd congruence number.
    Then one of the following is true
    \begin{enumerate}
        \item \label{bad} $E$ has a conductor $p$ and no $2$-torsion point,
            $E$ has supersingular reduction at $2$, and $E(\RR)$ is connected.
        \item $E$ has a conductor $p$ and a rational $2$-torsion point (hence
            it is a Neumann-Setzer curve), and $p=u^2+64$ with 
            $u \equiv 3 \pmod 8$.
        \item $E$ has a conductor $2p$ and $p=2^k-m^2$ for some odd integer 
            $7 \leq k$ and integer $m$,
            and $E$ is isogenous to
            \[ y^2+xy=x^3+{m-1 \over 4} x^2+2^{k-6}x.\]
        \item $E$ has a conductor $4p$ and $p=m^2+4$ for some integer
            $m \equiv 1 \pmod 4$, and $E$ is isogenous to one of
            \[   y^2 = x^3+mx^2-x.\]
        \item $E$ has a conductor $pq$ with $p$ and $q$ being odd primes,
            $p \equiv q \equiv 3 \pmod 8$, and for some odd integers $r$ and $s$,
            $p^r-q^s=16,$ and $E$ is isogenous to
            \[ y^2+xy=x^3+{p^r+15 \over 4}x^2 + p^r x.\]
        \item $E$ is one of the exceptional curves $11A$, $15A$, $17A$, $19A$, 
            $21A$, $24A$, $27A$, $32A$, $36A$, $37B$, $49A$,
            $243B$. 
    \end{enumerate}
    In all of the above cases, $E$ has rank $0$, except possibly
    in case \ref{bad}. In this case, we know that $E$ has an even analytic rank.
\end{thm}
Note that all of the curves in case 6 in the above theorem have a non-trivial
torsion point. Therefore we have proved that if $E$ has odd congruence number
and has a nontrivial torsion point, then it has rank $0$.
Also note that for all of the above cases, except for case \ref{bad}, we construct
a family of elliptic curves with all the desired torsion structures and conductors.
We expect
that all of these elliptic curves have odd congruence numbers. This can be proved 
if, for example,
we show that $J_0(N)[\mm] \rightarrow E[2]$ is injective and $J[\mm]=C_N[\mm]$.
When $E$ is a Neumann-Setzer curve, the results of \cite{Maz79} and \cite{MO89}
prove this result. 
We expect that similar results are true for the other cases;
however we, do not yet know of a proof of this result.

Finally, it is natural to ask how often do elliptic curves have odd 
congruence number. Since such elliptic curves can not have more than
three primes dividing their conductor, they are not that common. Furthermore
as soon as we have a nontrivial rational torsion point, we have a 
conjectural parametrization of all such elliptic curves. Therefore we like to
know how
often we get an optimal elliptic curve of prime conductor with no rational
torsion point having an odd modular degree. Looking through Cremona's
table of elliptic curves of conductor less than $130000$, we find 
$1991$ elliptic curves of prime conductor and trivial
rational torsion structure, out of which $196$ of those have an odd
modular degree.

    \bibliographystyle{plain}
    \bibliography{refs}
\end{document}